\documentclass{article}
\usepackage{cite}
\usepackage{amsmath,amssymb,amsthm}
\usepackage{titlesec,hyperref}
\usepackage{color}
\usepackage{fancyhdr}
\usepackage[margin=3cm]{geometry}

\newtheorem{theo}{Theorem}[section]
\newtheorem{lemm}[theo]{Lemma}
\newtheorem{defi}[theo]{Definition}

\newtheorem{prop}[theo]{Proposition}
\newtheorem{rema}[theo]{Remark}
\numberwithin{equation}{section}
\allowdisplaybreaks
\begin{document}
	\title{Global Gevery regulartiy  and analyticity of a weakly dissipative Camassa-Holm equation}
	\author{
		Zhiying $\mbox{Meng}^1$ \footnote{email: mengzhy3@mail2.sysu.edu.cn} \quad and\quad
		Zhaoyang $\mbox{Yin}^{1,2}$ \footnote{email: mcsyzy@mail.sysu.edu.cn}\\
		$^1\mbox{Department}$ of Mathematics, Sun Yat-sen University,\\
		Guangzhou, 510275, China\\
		$^2\mbox{Faculty}$ of Information Technology,\\
		Macau University of Science and Technology, Macau, China
	}
	\date{}
	\maketitle
	\begin{abstract}
		This work is concerned with the  Gevrey regularity and analyticity of the solution to a weakly dissipative Camassa-Holm system. We first demonstrate the local Gevery regularity and analyticity of this  equation. Then, we disscuss the  continuity of the data-to-solution map. Finally, we obtain the global Gevery regularity of this system in Gevery class $G_{\sigma}$ with $\sigma\geq 1$ in time.
		\end{abstract}
	\tableofcontents
	\section{Introduction}
	\par\
	
	In this paper, we study the following weakly dissipative Camassa-Holm equation \cite{Igor2020jde}
	\begin{equation}\label{0.1}
		\left\{\begin{array}{l}
			u_t-u_{txx}+3uu_x+\lambda(u-u_{xx})=2uu_x+uu_{xxx}+\alpha u+\beta u^2u_x+\gamma u^3u_x+\Gamma u_{xxx},~t>0, \\
			u(0,x)=u_0,~x\in\mathbb{R},
		\end{array}\right.
	\end{equation}
where $\alpha, \beta,\gamma,\Gamma$ are real number and $\lambda$ is a positive number. Let $\Lambda ^{-2}=(1-\partial_{xx})^{-1},~h(u)=(\alpha+\Gamma)u+\frac{\beta}{3}u^3+\frac{\gamma}{4}u^4.$  The above equation can be reformulted in the following form
	\begin{equation}\label{u}
	\left\{\begin{array}{l}
		u_t+(u+\Gamma)u_x+\lambda u=Q,\ t>0, ~~x\in\mathbb{R},\\
		u(0,x)=u_0,
	\end{array}\right.
\end{equation}
where $Q=-\Lambda ^{-2}\partial_x\Big(-h(u)+u^2+\frac{1}{2}u_x^2\Big).$
 For  $\lambda=\alpha=\beta=\gamma=\Gamma=0,$ the equation \eqref{u} reads as the Camassa-Holm (CH) equation
 \begin{align}\label{ch}
 	(1-\partial_x^2)u_t=3uu_x-2u_xu_{xx}-uu_{xxx},
 \end{align}
which has a bi-Hamiltonian structure  and is completely integrable
 \cite{Camassa1993,Constantin01scat}. The local well-posedness of CH equation for initial data $u_0\in H^s$ with $s>\frac 3 2$ has been proved in \cite{Constantin1998,R2001,Constantin1998whs}.  In addition, many researchers established the local well-posedness of CH equation in Besov spaces $B^s_{p,r}$ with $({\rm i})~s>1+\frac{1}{p};~({\rm ii})~ s={1+\frac{1}{p}},~p\in[1,\infty),r=1$ \cite{LiA2000jde,Danchin2003wp,Danchin2001inte,Li2016nwpC,YE2021A}. The CH equation has not only global strong solutions \cite{Constantingeometric,Constantin1998,Constantin1998whs} but also the solutions will blow up in finite time \cite{Bressan2006g-c,Constantin1998whs,Constantin1998}. The authors established the existence of the unique global solutions \cite{Bressan2015character}. In \cite{Bressan2007gd}, Bressan  and Constantin  proved the global  dissipative solutions  of the CH equation. Its local  ill-posedness problem was presented in \cite{Guoyy2021Ill,Guo2019ill,Li2021}. The analyticity for the solutions of CH equation were studied in \cite{Barostichi16,Himonasana03}.  In \cite{LWgevery}, the authors discussed the Gevery regularity for the Camassa-Holm  type systems by a generalized  Ovsyannikov theorem.

Recently, the author established the local well-posedness of system \eqref{u}  in Sobolev spaces $H^s$ with $s>\frac 3 2$ \cite{Igor2020jde}. In \cite{MengYin}, the  authors proved that this equation is locally well-posed in Besov spaces  and the existence of global strong solutions under the condition that small initial  datum. And the existence and uniqueness of weak solutions were presented in \cite{MY}. The analyticity and Gevery regularity of the system \eqref{u} has not been investigated. Hence, in this paper, using the generalized  Ovsyannikov theorem \cite{LWgevery}, we get the  local analyticity and Gevrey regularity of the solutions to system \eqref{u}. And we see that  the continuity of the data-to-solution map.  On the other hand, following the idea \cite{Levermoreeular}, we employ the global analyticity and Gevrey regularity of this system.

Our  paper is organized as follows. In Section 2, we give some preliminary results. In Sections 3, 4, we discuss  the   the local analyticity and Gevrey regularity, and the continuity of the data-to-solution map  of the equation \eqref{u}. In Section 5, we study the   global analyticity and Gevrey regularity.
	\section{Preliminaries}
	\par\
	
	In this section, we give some theorems and some lemmas that will be of use to prove our main results. Now,
		we study the  Cauchy problem for the above system which can be rewritten  in the following abstract form
			\begin{equation}\label{cauchy}
			\left\{\begin{aligned}
				&\frac{du}{dt}=F(t,u(t)),\\
				&u(0,x)=u_0.
			\end{aligned}\right.
		\end{equation}
	
	\begin{theo}\cite{O1965Singular,Baouendi77G,NirenbergG72}\label{O}
	Let $\{X_{\delta}\}_{0<\delta<1}$  be a scale of decreasing Banach spaces, namely, for any $0<\delta'<\delta$, we have $X_\delta\subset X_{\delta'}$ and $\|\cdot\|_{\delta'}\leq\|\cdot\|_{\delta}.$ Given $T,~{R}>0.$ Let $u_0\in X_1,$ assume that\\
	(1) For any $0<\delta'<\delta<1,$ the function $t\mapsto u(t)$ is holomorphic in $|t|< T$ and continuous on $|t|< T$ with values in $X_{\delta}$ and
	$$\sup_{|t|< T}\|u(t)\|_{\delta}<R.$$
	Then $t\mapsto F(t,u(t))$ is a holomorphic function on $|t|<T$ with values in $X_{\delta'}.$\\
	(2)For any $0<\delta'<\delta<1,$ and any $u,~v\in \overline{B(u_0,R)}\subset X_{\delta},$ there exists a positive constant $L$ depending
on $u_0$ and $R$ such that
	$$\sup_{|t|< T}\|F(t,u(t))-F(t,v(t))\|_{\delta'}\leq\frac{L}{\delta'-\delta}\|u-v\|_{\delta}.$$\\
	(3)For  any $0<\delta<1,$ there exists a positive constant $M$ depending on $u_0$ such that
	$$\sup_{|t|< T}\|F(t,u_0(t))\|_{\delta'}\leq\frac{M}{1-\delta}.$$
	Then there exist a
	$T_0\in(0,T)$ and a unique solution to the Cauchy problem \eqref{cauchy}, which for any 	$\delta\in (0,1)$ is  holomorphic in $|t|<T_0(1-\delta)$ with values in $X_{\delta}.$
\end{theo}
The above theorem was first put forwarded by Ovsyannikov in \cite{O1965Singular,Ononlocal,Ononlinear71}. Note that the  original Ovsyannikov
theorem is not valid for  the Gervey class. Owing to this kind of space  cannot such that the condition (2)  hold in Theorem \ref{cauchy}. Hence, for the Gevery class, we prove that
	$$\sup_{|t|< T}\|F(t,u(t))-F(t,v(t))\|_{\delta'}\leq\frac{L}{(\delta'-\delta)^\sigma}\|u-v\|_{\delta},$$
	where $\sigma\geq 1.$ As $\sigma>1,$  the above  inequality is weaker than the condition (2).
Recently,  Luo and Yin build a new auxiliary function to get a generalized  Ovsyannikov
theorem by modifying the proof of \cite{LWgevery}.
\begin{theo}\cite{LWgevery,Zhanglei}\label{L}
	Let $\{X_{\delta}\}_{0<\delta<1}$  be a scale of decreasing Banach spaces, namely, for any $0<\delta'<\delta$, we have $X_\delta\subset X_{\delta'}$ and $\|\cdot\|_{\delta'}\leq\|\cdot\|_{\delta}.$ Given $T,~{R}>0.$ Let $u_0\in X_1,$ assume that\\
	(1) For any $0<\delta'<\delta<1,$ the function $t\mapsto u(t)$ is holomorphic in $|t|< T$ and continuous on $|t|< T$ with values in $X_{\delta}$ and
	$$\sup_{|t|< T}\|u(t)\|_{\delta}<R.$$
	Then $t\mapsto F(t,u(t))$ is a holomorphic function on $|t|<T$ with values in $X_{\delta'}.$\\
	(2) For any $0<\delta'<\delta<1,$ and any $u,~v\in \overline{B(u_0,R)}\subset X_{\delta},$ there exists a positive constant $L$ depending
	on $u_0$ and $R$ such that
	$$\sup_{|t|< T}\|F(t,u(t))-F(t,v(t))\|_{\delta'}\leq\frac{L}{(\delta'-\delta)^\sigma}\|u-v\|_{\delta}.$$\\
	(3) For  any $0<\delta<1,$ there exists a positive constant $M$ depending on $u_0$ such that
	$$\sup_{|t|< T}\|F(t,u_0(t))\|_{\delta'}\leq\frac{M}{1-\delta}.$$
	Then there exists a
	$T_0\in(0,T)$ and a unique solution to the Cauchy problem \eqref{cauchy}, which for any 	$\delta\in (0,1)$ is  holomorphic in $|T|<\frac{T_0(1-\delta)^\sigma D_{\sigma} }{2^\sigma-1}$ with values in $X_{\delta}$   and $D_{\sigma}=\frac{1}{2^{\sigma}-2+\frac{1}{2^{\sigma+1}}}.$
\end{theo}

Now, we give a  Banach spaces, which can be used to prove Theorem \ref{L}  by using the fixed point argument.
\begin{defi}\cite{LWgevery,Zhanglei}\label{defi E}
	Let $\sigma\geq 1.$ For any $a>0,$ we represent by $E_a$  the space of functions $u(t)$ which for each $0<\delta<1$ and $|t|<\frac{a(1-\delta)^\sigma  D_{\sigma}}{2^\sigma-1}$ with $D_{\sigma}=\frac{1}{2^{\sigma}-2+\frac{1}{2^{\sigma+1}}}$ are holomorphic  and continuous functions of $t$ with values in $X_{\delta}$ such that
	$$\|u\|_{E_T}:=\sup_{|t|<\frac{a(1-\delta)^\sigma}{2^\sigma-1}}\Big(\|u(t)\|_{G^{\delta}_{\sigma,s}}(1-\delta)^{\sigma}\sqrt{1-\frac{|t|}{a(1-\delta)^{\sigma}}}\Big)<+\infty.$$
\end{defi}

\begin{rema}\cite{LWgevery,Zhanglei}\label{t}
	Indeed, $T_0={\min}\{\frac{1}{2^{2\sigma+4}L},~\frac{(2^{\sigma}-1)R}{(2^{\sigma}-1)2^{2\sigma+3}LR+MD_{\sigma}}\}$  with $D_{\sigma}=\frac{1}{2^{\sigma}-2+\frac{1}{2^{\sigma+1}}}$, which provides a lower bound of the lifespan.
\end{rema}
\begin{rema}\label{t1}
	If $\sigma=1,$ Theorem \ref{cauchy} reduces to the so-called abstract Cauchy-Kovalevsky theorem.
\end{rema}
Now, we give the definition of Sobolev-Gevery spaces and recall some properties are often used.
\begin{defi}\cite{FoiasGEVERYNS}\label{defi gevery}
	Let $\sigma,~\delta>0$ and $s$ be  a real number. A function $f\in G^{\delta}_{\sigma,s}(\mathbb{R}^d)$ if and only of $f\in \mathcal{C}^{\infty}(\mathbb{R}^d)$ and satisfies
	$$\|f\|_{G^{\delta}_{\sigma,s}}=\Big(\int_{\mathbb{R}}(1+|\xi|^2)^s e^{2\delta(1+|\xi|^2)^{\frac{1}{2\sigma}}}|\hat{f}(\xi)|^2d\xi\Big)^{\frac 1 2}<\infty.$$
	A function $f\in G^{\delta}_{\sigma,s}(\mathbb{R}^d)$ if and only of $f\in \mathcal{C}^{\infty}(\mathbb{R}^d)$ and satisfies
	$$\|f\|_{\bar{G}^{\delta}_{\sigma,s}}=\Big(\int_{\mathbb{R}}(1+|\xi|^2)^s e^{2\delta|\xi|^{\frac{1}{\sigma}}}|\hat{f}(\xi)|^2d\xi\Big)^{\frac 1 2}<\infty.$$
\end{defi}
\begin{defi}\cite{MM}\label{class}
	A function is of Gevrey class $\sigma\geq 1$, if there exist $\delta>0,~s\geq 0$ such that  $f\in G^{\delta}_{\sigma,s}(\mathbb{R}^d).$
	Denote the the functions of Gevrey class $\sigma$
	$$f\in G_{\sigma}(\mathbb{R}^d):=\cup_{\delta>0.s\in \mathbb{R}}G^{\delta}_{\sigma,s}(\mathbb{R}^d).$$
\end{defi}
\begin{rema}\cite{hehuij}\label{oper}
	Denote the Fourier multiplier $e^{\delta{\Lambda}^{\frac{1}{\sigma}}}$ and $e^{\delta(-\Delta)^{\frac{1}{2\sigma}}}$ by\\
	
		$~~~~~~~~~~~~~~~~~$$e^{\delta{\Lambda}^{\frac{1}{\sigma}}}f=\mathcal{F}^{-1}\Big(e^{\delta(1+|\xi|^2)^{\frac{1}{2\sigma}}}\hat{f}\Big)$ $~~$and$~~$ $e^{\delta(-\Delta)^{\frac{1}{2\sigma}}}f=\mathcal{F}^{-1}\Big(e^{\delta|\xi|^{\frac{1}{\sigma}}}\hat{f}\Big).$	$~~~~~~~~~~~~~~~~~~~~~~~~~~$
	Then we have $\|f\|_{G^{\delta}_{\sigma,s}(\mathbb{R}^d)}=\|e^{\delta{\Lambda}^{\frac{1}{\sigma}}}f\|_{H^s(\mathbb{R}^d)}$ and $\|f\|_{\bar{G}^{\delta}_{\sigma,s}}=\|e^{\delta(-\Delta)^{\frac{1}{2\sigma}}}f\|_{H^s(\mathbb{R}^d)}.$ Moreover, for  any $\sigma\geq 1,$ we get
	$$\|f\|_{\bar{G}^{\delta}_{\sigma,s}}\leq \|f\|_{{G}^{\delta}_{\sigma,s}} \leq e^{\delta}\|f\|_{\bar{G}^{\delta}_{\sigma,s}}.$$ If $0<\sigma<1,$ it is called ultra-analytic function. If $\sigma=1,$ it is usual analytic function and $\delta$ is called the radius of analyticity. If $\delta>1,$ it is the Gevery class function.
\end{rema}
\begin{prop}\label{embedding}
	If  $0<\delta'<\delta,~0<\sigma'<\sigma$ and $s'<s.$ Then we have $G^{\delta}_{\sigma,s}(\mathbb{R}^d)\hookrightarrow G^{\delta'}_{\sigma,s}(\mathbb{R}^d),~G^{\delta}_{\sigma',s}(\mathbb{R}^d)\hookrightarrow G^{\delta}_{\sigma,s}(\mathbb{R}^d)$ and $G^{\delta}_{\sigma,s}(\mathbb{R}^d)\hookrightarrow G^{\delta}_{\sigma,s'}(\mathbb{R}^d)$.
\end{prop}
\begin{prop}\cite{LWgevery}\label{d'd}
	Let $f\in G^{\delta}_{\sigma,s}(\mathbb{R}).$ Assume that  $0<\delta'<\delta,~\sigma>0$ and $s\in \mathbb{R}$. Then we have
		\begin{align*}
			\|\partial_{x}f\|_{G^{\delta'}_{\sigma,s}(\mathbb{R})}&\leq \frac{e^{-\sigma}\sigma^{\sigma}}{2(\delta'-\delta)^{\sigma}}\|f\|_{G^{\delta}_{\sigma,s}(\mathbb{R})},\\
		\|(1-\partial_{xx})^{-1}f\|_{G^{\delta'}_{\sigma,s}(\mathbb{R})}&\leq \|f\|_{G^{\delta}_{\sigma,s-2}(\mathbb{R})}\leq \|f\|_{G^{\delta}_{\sigma,s}(\mathbb{R})},\\
		\|(1-\partial_{xx})^{-1}\partial_{x}f\|_{G^{\delta'}_{\sigma,s}(\mathbb{R})}&\leq \|f\|_{G^{\delta}_{\sigma,s-1}(\mathbb{R})}\leq \|f\|_{G^{\delta}_{\sigma,s}(\mathbb{R})}.
	\end{align*}
\end{prop}

\begin{prop}\label{fg}
	Let $s>\frac 1 2,~\sigma\geq 1$ and $\delta>0.$ Then $G^{\delta}_{\sigma,s}(\mathbb{R})$ is an algebra. Moreover, there exists two constants $C_s$ and $\bar{C}_s$ such that
	\begin{align*}	
	\|fg\|_{G^{\delta}_{\sigma,s}(\mathbb{R})}&\leq C_s\|f\|_{G^{\delta}_{\sigma,s}(\mathbb{R})}\|g\|_{G^{\delta}_{\sigma,s}(\mathbb{R})},\\
	\|fg\|_{G^{\delta}_{\sigma,s-1}(\mathbb{R})}&\leq \bar{C}_s\|f\|_{G^{\delta}_{\sigma,s}(\mathbb{R})}\|g\|_{G^{\delta}_{\sigma,s-1}(\mathbb{R})}.
\end{align*}
\end{prop}

\begin{lemm}\cite{LWgevery,Zhanglei}\label{Ea}
	Let $\sigma\geq 1.$ For each $a>0,~u\in E_a,~0<\delta<1$ and $0\leq t<\frac{a(1-\delta)^\sigma D_{\sigma}}{2^\sigma-1}$ with $D_{\sigma}=\frac{1}{2^{\sigma}-2+\frac{1}{2^{\sigma+1}}}.$ Then we have
	$$\int_0^t \frac{\|u(\tau)\|_{\delta(\tau)}}{(\delta(\tau)-\delta)^{\sigma}}\leq \frac{a2^{2\delta+3}	\|u\|_{E_a}}{(1-\delta)^\sigma}\sqrt{\frac{a(1-\delta)^{\sigma}}{a(1-\delta)^{\sigma}-t}},$$
	where $$\delta(\tau)=\frac 1 2 (1+\delta)+(\frac 1 2)^{2+\frac{1}{\sigma}}\Big([(1-\delta)^{\sigma}-\frac t a]^{\frac{1}{\sigma}}-[(1-\delta)^{\sigma}+(2^{\sigma+1}-1)\frac t a]^{\frac{1}{\sigma}}\Big)\in (0,1).$$
\end{lemm}
\begin{lemm}\cite{MengYin}\label{hsglobal}
		Let $u_0 \in B^{s}_{p,r}$ with  $s>\frac{3}{2},~p=r=2.$  A  constant $\epsilon$ such that if
	\begin{align}\label{ss0}
		H_0	\triangleq|\alpha|+|\Gamma|+\|u_0\|_{B^s_{p,r}}+\frac{|\beta|}{3}\|u_0\|^2_{B^s_{p,r}}+\frac{|\gamma|}{4}\|u_0\|^3_{B^s_{p,r}}\leq \lambda \epsilon,
	\end{align}
	for all real numbers $\alpha, \beta, \gamma, \Gamma$ and $\lambda>0.$ 
	Then the solution to the system \eqref{u}  exists globally in time. Moreover, for any $t\in [0,\infty)$ we have
	\begin{align}\label{small}
		H (t)\triangleq|\alpha|+|\Gamma|+\|u(t)\|_{B^s_{p,r}}+\frac{|\beta|}{3}\|u(t)\|^2_{B^s_{p,r}}+\frac{|\gamma|}{4}\|u(t)\|^3_{B^s_{p,r}}\leq H_0.
	\end{align} 
\end{lemm}
	\section{Local Gevery regularity and analyticity}
\par\

	This section is devoted to establishing the local Gevery regularity and analyticity.
	Our main results can be stated as follows.
	\begin{theo}\label{local}
		Let $\sigma\geq 1,~s>\frac 3 2.$ Assume that $u_0\in G^1_{\sigma,s}.$ Then  for any $0<\delta<1,$ there exists a $T_0>0$ such that the system \eqref{u} has a unique solution $u$ which is holomorphic in $|t|<\frac{T_0(1-\delta)^{\sigma}}{2^{\sigma}-1}$ with values in $G^{\delta}_{\sigma,s}(\mathbb{R}).$ Moreover $T_0= \frac{1}{2^{2\sigma+8}C'(e^{-\sigma}\sigma^{\sigma}+2)(1+\|u_0\|_{G^{1}_{\sigma,s}})^4},$ where the positive constant $C'$ depending on $s,~\alpha,~\beta,~\gamma,~\lambda,~\Gamma.$
	\end{theo}
\begin{proof}
	The equation \eqref{u} can be written as follows
	\begin{equation}\label{k1}
	\left\{\begin{aligned}
		&\frac{d}{dt}u=F(t,u(t)),\\
		&u(0,x)=u_0,
		\end{aligned}\right.
	\end{equation}
where $F(t,u(t))=-(u+\Gamma)u_x-\lambda u-P_x.$
Fixed a $\sigma\geq1$ and $s>\frac 3 2.$  Theorem \ref{L} entails that $\{G^{\delta}_{\sigma,s}\}_{0>\delta<1}$ is a scale of decreasing
Banach spaces. For any $0<\delta'<\delta,$  it follows from  Propositions \ref{embedding}-\ref{fg} that
\begin{align}\label{e0}
	\|F(t,u(t))\|_{G^{\delta'}_{\sigma,s}}&\leq \frac 1 2\|(u^2)_x\|_{G^{\delta'}_{\sigma,s}}+ +|\Gamma|\|u_x\|_{{G^{\delta'}_{\sigma,s}}}+C_s\Big(\|h(u)\|_{G^{\delta'}_{\sigma,s-1}}+\|u^2\|_{G^{\delta'}_{\sigma,s-1}}+\|u_x^2\|_{G^{\delta'}_{\sigma,s-1}}\Big)\notag\\
	&\leq \frac{C_se^{-\sigma}\sigma^{\sigma}}{2(\delta'-\delta)^{\sigma}}\|u^2\|_{G^{\delta}_{\sigma,s}}+\frac{C_{s,\Gamma}e^{-\sigma}\sigma^{\sigma}}{(\delta'-\delta)^{\sigma}}\|u\|_{G^{\delta}_{\sigma,s}}+\|P_x\|_{G^{\delta'}_{\sigma,s}}.
\end{align}
Using Propositions \ref{embedding}-\ref{fg} again, we have
\begin{align*}
	\|P_x\|_{G^{\delta'}_{\sigma,s}}&\leq C_s\Big(\|h(u)\|_{G^{\delta'}_{\sigma,s-1}}+\|u^2\|_{G^{\delta'}_{\sigma,s-1}}+\|u_x^2\|_{G^{\delta'}_{\sigma,s-1}}\Big)\notag\\&\leq C'\Big(\|u\|_{G^{\delta}_{\sigma,s}}+\|u\|^2_{G^{\delta}_{\sigma,s}}+\|u\|^3_{G^{\delta}_{\sigma,s}}+\|u\|^4_{G^{\delta}_{\sigma,s}}+\|u\|^5_{G^{\delta}_{\sigma,s}}\Big)\notag\\
	&\leq C'\|u\|_{G^{\delta}_{\sigma,s}}(1+\|u\|_{G^{\delta}_{\sigma,s}})^4,
\end{align*}
where $C'$ depending on $s,\alpha,\beta,\gamma,\lambda,\Gamma.$\\
Adding up the above estimate to \eqref{e0}, one can deduce that
\begin{align}\label{e1}
		\|F(t,u(t))\|_{G^{\delta'}_{\sigma,s}}&\leq
		\frac{C'(e^{-\sigma}\sigma^{\sigma}+2)}{2(\delta'-\delta)^{\sigma}}\|u\|_{G^{\delta}_{\sigma,s}}(1+\|u\|_{G^{\delta}_{\sigma,s}})^4,
\end{align}
which means that $F(t,u(t))$ satisfies the condition (1) of Theorem \ref{cauchy}. Adopting the similar procedure for  $F(u_0)$ as above, we arrive at
\begin{align*}
		\|F(u_0)\|_{G^{\delta}_{\sigma,s}}&\leq
	\frac{C'(e^{-\sigma}\sigma^{\sigma}+2)}{2(1-\delta)^{\sigma}}\|u_0\|_{G^{1}_{\sigma,s}}(1+\|u_0\|_{G^{1}_{\sigma,s}})^4.
\end{align*}
Therefore, we obtain that $F(t,u(t))$ satisfies the condition (3) in Theorem \ref{cauchy} with $$M=\frac{C'(e^{-\sigma}\sigma^{\sigma}+2)}{2}\|u_0\|_{G^{1}_{\sigma,s}}(1+\|u_0\|_{G^{1}_{\sigma,s}})^4.$$
On the other hand, we have to check that $F(t,u(t))$ satisfies the condition (2) of Theorem \ref{cauchy}. Let $\|u-u_0\|_{G^{\delta}_{\sigma,s}}\leq R$ and $\|v-v_0\|_{G^{\delta}_{\sigma,s}}\leq R.$ One can get from Theorem \ref{cauchy} and Propositions \ref{embedding}-\ref{fg} that
\begin{align}\label{e2}
	\|F(u)-F(v)\|_{G^{\delta'}_{\sigma,s}}&\leq \frac{C_se^{-\sigma}\sigma^{\sigma}}{2(\delta'-\delta)^{\sigma}}\|u^2-v^2\|_{G^{\delta}_{\sigma,s}}+\frac{C_{s,\Gamma}e^{-\sigma}\sigma^{\sigma}}{(\delta'-\delta)^{\sigma}}\|u-v\|_{G^{\delta}_{\sigma,s}}+\|P_x(u)-P_x(v)\|_{G^{\delta'}_{\sigma,s}}\notag\\&\leq \frac{C_se^{-\sigma}\sigma^{\sigma}}{2(\delta'-\delta)^{\sigma}}\|u-v\|_{G^{\delta}_{\sigma,s}}\|u+v\|_{G^{\delta}_{\sigma,s}}+\frac{C_{s,\Gamma}e^{-\sigma}\sigma^{\sigma}}{(\delta'-\delta)^{\sigma}}\|u-v\|_{G^{\delta}_{\sigma,s}}\notag\\
	&~~~+C'\|u-v\|_{G^{\delta}_{\sigma,s}}\Big(1+\|u+v\|_{G^{\delta}_{\sigma,s}}\Big)^4\notag\\
	&\leq 	\frac{C_se^{-\sigma}\sigma^{\sigma}}{2(\delta'-\delta)^{\sigma}}\|u-v\|_{G^{\delta}_{\sigma,s}}(\|u_0\|_{G^{\delta}_{\sigma,s}}+R+1)+C'\|u-v\|_{G^{\delta}_{\sigma,s}}(1+\|u_0\|_{G^{1}_{\sigma,s}}+R)^4\notag\\
	&\leq 	\frac{C'(e^{-\sigma}\sigma^{\sigma}+2)}{2(\delta'-\delta)^{\sigma}} \|u-v\|_{G^{\delta}_{\sigma,s}}(1+\|u_0\|_{G^{1}_{\sigma,s}}+R)^4.
\end{align}
Hence, we get $F$ satisfies the condition (2) of Theorem \ref{cauchy} with $$L={C'(e^{-\sigma}\sigma^{\sigma}+2)}(\|u_0\|_{G^{1}_{\sigma,s}}+R+1)^4.$$
Then, we  conclude that the  local existence result of \eqref{local} with the Gevrey regularity or analyticity, and
$$T_0={\min}\{\frac{1}{2^{2\sigma+4}L},~\frac{(2^{\sigma}-1)R}{(2^{\sigma}-1)2^{2\sigma+3}LR+MD_{\sigma}}\},$$
 with $D_{\sigma}=\frac{1}{2^{\sigma}-2+\frac{1}{2^{\sigma+1}}}.$

On the other hand, choosing $R=1+\|u_0\|_{G^{1}_{\sigma,s}},$ we infer that $L=2^4{C'(e^{-\sigma}\sigma^{\sigma}+2)}(1+\|u_0\|_{G^{1}_{\sigma,s}})^4$ and  $M\leq 2^{2\sigma+3}LR$. Moreover, we get
$$T_0=\frac{1}{2^{2\sigma+8}C'(e^{-\sigma}\sigma^{\sigma}+2)(1+\|u_0\|_{G^{1}_{\sigma,s}})^4}.$$

Thus, we finish the proof of Theorem \ref{local}.
\end{proof}
\section{Continuity of the data-to-solution map}
\par\

 In this section, we consider the  continuity of the data-to-solution map for initial data and solutions in Theorem \ref{local}. We first introduce a definition to illustrate  what means the data-to-solution map is continuous
 from $G^1_{\sigma,s}(\mathbb{R})$ into the solutions space.
\begin{defi}\label{defi continue}
	Let $\sigma\geq 1$ and $s>\frac 3 2.$ Then, the data-to-solution map  $u_0\mapsto u$ of the system  \eqref{u} is continuous if for a given initial data $u_0^{\infty}\in G^1_{\sigma,s}$ there exists a  $T=T(\|u_0\|_{G^1_{\sigma,s}},\|u_0^{\infty}\|_{G^1_{\sigma,s}})>0,$ such that for any sequence $u_0^n\in G^1_{\sigma,s}$ and $\|u_0^n-u_0^{\infty}\|_{G^1_{\sigma,s}}\to 0$ as $n\to \infty,$ the corresponding solutions $u^n$ of system \eqref{u} satisfy $\|u^n-u^{\infty}\|_{E_T}\to 0$ $n\to \infty,$ where
	$$\|u\|_{E_T}:=\sup_{|t|<\frac{T(1-\delta)^\sigma}{2^\sigma-1}}\Big(\|u(t)\|_{G^{\delta}_{\sigma,s}}(1-\delta)^{\sigma}\sqrt{1-\frac{|t|}{T(1-\delta)^{\sigma}}}\Big).$$
\end{defi}
\begin{theo}\label{theo continue}
	Given $\sigma\geq 1$ and $s>\frac 3 2.$ Let $u_0\in G^1_{\sigma,s}(\mathbb{R}).$ Then the data-to-solution map   $u_0\mapsto u$ of the system \eqref{u} is continuous from $G^1_{\sigma,s}(\mathbb{R})$
	into the solutions space.
\end{theo}
\begin{proof}
	In general, assume that $t\geq 0$. Defining
	$$T^{\infty}=\frac{1}{2^{2\sigma+8}C'(e^{-\sigma}\sigma^{\sigma}+2)(1+\|u^{\infty}_0\|_{G^{1}_{\sigma,s}})^4},~T^{n}=\frac{1}{2^{2\sigma+8}C'(e^{-\sigma}\sigma^{\sigma}+2)(1+\|u^n_0\|_{G^{1}_{\sigma,s}})^4}.$$
	Owing to $\|u_0^n-u_0^{\infty}\|_{G^1_{\sigma,s}}\to 0$ as $n\to \infty,$  there exists a  constant $N$ such that if $n\geq N$, we obtain
	\begin{align}\label{c0}
		\|u_0^n\|_{G^1_{\sigma,s}}\leq \|u_0^{\infty}\|_{G^1_{\sigma,s}}+1.
	\end{align}
Set
\begin{align}\label{c1}
	T=\frac{1}{2^{2\sigma+8}C'(e^{-\sigma}\sigma^{\sigma}+2)(2+\|u^{n}_0\|_{G^{1}_{\sigma,s}})^4}.
\end{align}
This means that $T\leq\min\{T^n,~T^{\infty}\}$ for any $n\geq N.$ A similar argument based on the proof of Theorem \ref{local}, we claim that $T^n$ and $T^{\infty}$ are the existence time of the solutions $u^n$ and $u^{\infty}$ corresponding to $u_0^n$  and $u_0^{\infty}$
respectively.Therefore, we see that for any $n\geq N,$
\begin{align}
	u^{\infty}(t,x)&=u_0^{\infty}+\int_0^tF(\tau,	u^{\infty}(t,\tau))d\tau,~0\leq t< \frac{T(1-\delta)^\sigma}{2^\sigma-1}, \\
	u^{n}(t,x)&=u_0^{n}~+\int_0^tF(\tau,	u^{n}(t,\tau))d\tau,~~0\leq t< \frac{T(1-\delta)^\sigma}{2^\sigma-1},
\end{align}
with $F$ is given in Theorem \ref{local}.Thus, for any $0\leq t\leq \frac{T(1-\delta)^\sigma}{2^\sigma-1}$ and $0\leq \delta \leq 1,$ we can obtain
\begin{align}\label{c3}
	\|u^{n}(t)-u^{\infty}(t)\|_{\delta}\leq \|u_0^{n}(t)-u_0^{\infty}(t)\|_{\delta}+\int_0^t\|F(u^n(\tau))-F(u^{\infty}(\tau))\|_{\delta}d\tau.
\end{align}
Define that
$$\delta(t)=\frac 1 2 (1+\delta)+(\frac 1 2)^{2+\frac{1}{\sigma}}\Big([(1-\delta)^{\sigma}-\frac t T]^{\frac{1}{\sigma}}-[(1-\delta)^{\sigma}+(2^{\sigma+1}-1)\frac t T]^{\frac{1}{\sigma}}\Big).$$
Making use of Lemma \ref{Ea}, we know that $\delta<\delta(t)<1.$ One can get from \eqref{c3} that
\begin{align}\label{c4}
	\|F(u^n(\tau))-F(u^{\infty}(\tau))\|_{\delta}\leq \frac{L\|u^{n}(\tau)-u^{\infty}(\tau)\|_{\delta(\tau)}}{(\delta(\tau)-\delta)^{\sigma}},
\end{align}
where $L=2^4{C'(e^{-\sigma}\sigma^{\sigma}+2)}(1+\|u_0\|_{1})^4.$ Hence, we can deduce that
\begin{align}\label{c5}
	\|u^{n}(t)-u^{\infty}(t)\|_{\delta}\leq \|u_0^{n}(t)-u_0^{\infty}(t)\|_{\delta}+L\int_0^t\frac{\|u^{n}(\tau)-u^{\infty}(\tau)\|_{\delta(\tau)}}{(\delta(\tau)-\delta)^{\sigma}}d\tau.
\end{align}
It follows from Lemma \ref{Ea} that
\begin{align}\label{c6}
		\|u^{n}(t)-u^{\infty}(t)\|_{\delta}\leq \|u_0^{n}(t)-u_0^{\infty}(t)\|_{\delta}+\frac{T2^{2\delta+3}	\|u^{n}(t)-u^{\infty}(t)\|_{E_T}}{(1-\delta)^\sigma}\sqrt{\frac{T(1-\delta)^{\sigma}}{T(1-\delta)^{\sigma}-t}}.
\end{align}
Combining \eqref{c1} and $LT2^{{2\sigma+3}}<\frac 1 2$. Therefore,  we obtain
\begin{align*}
	\|u^{n}(t)-u^{\infty}(t)\|_{\delta}\leq \|u_0^{n}(t)-u_0^{\infty}(t)\|_{\delta}+\frac{	\|u^{n}(t)-u^{\infty}(t)\|_{E_T}}{2(1-\delta)^\sigma}\sqrt{\frac{T(1-\delta)^{\sigma}}{T(1-\delta)^{\sigma}-t}},
\end{align*}
which implies
\begin{align*}
	&\|u^{n}(t)-u^{\infty}(t)\|_{E_T}(1-\delta)^{\sigma}\sqrt{1-\frac{t}{T(1-\delta)^{\sigma}}}\notag\\
	&\leq \|u_0^{n}-u_0^{\infty}\|_{\delta}(1-\delta)^{\sigma}\sqrt{1-\frac{t}{T(1-\delta)^{\sigma}}}+\frac 1 2 \|u^{n}(t)-u^{\infty}(t)\|_{E_T}\notag\\
	&\leq \|u_0^{n}-u_0^{\infty}\|_{1}+\frac 1 2 \|u^{n}(t)-u^{\infty}(t)\|_{E_T}.
\end{align*}
In view of the right hand side of the above estimate  is independent of $t$ and $\delta.$ By taking the supermum over $0<\delta<1,~0<t< \frac{T(1-\delta)^\sigma}{2^\sigma-1},$ we see that
\begin{align*}
	\|u^{n}(t)-u^{\infty}(t)\|_{E_T}\leq \|u_0^{n}-u_0^{\infty}\|_{1}+\frac 1 2 	\|u^{n}(t)-u^{\infty}(t)\|_{E_T},
\end{align*}
which means that
\begin{align*}
	\|u^{n}(t)-u^{\infty}(t)\|_{E_T}\leq 2	\|u^{n}(t)-u^{\infty}(t)\|_{E_T}.
\end{align*}
This verifies the above inequality is valid for any $n\geq N.$
\end{proof}
	\section{Global Gevery regularity and analyticity}
	\par\
	
	In this section, we focus on the global Gevery regularity of solution to the system \eqref{u}.
	Before stating the main result of this section, we introduce some lemmas.
	\begin{lemm}\cite{hehuij}\label{polo}
		Let $\delta\geq 0,~\sigma\geq 1$ and $s>1.$ Then, for any $\xi,\eta\in\mathbb{R}$, we have
		\begin{align*}
			&|(1+\xi^2)^{\frac s 2}e^{\delta(1+\xi^2)^{\frac{1}{2\sigma}}}-(1+\eta^2)^{\frac s 2}e^{\delta(1+\eta^2)^{\frac{1}{2\sigma}}}|\notag\\
			&\leq C_s|\xi-\eta|\Big((1+|\xi-\eta|^2)^{\frac{s-1}{2}}+(1+|\eta|^2)^{\frac{s-1}{2}}\notag\\&~~~+\delta[(1+|\xi-\eta|^2)^{\frac{s-1}{2} +\frac{1}{2\sigma}}+(1+|\eta|^2)^{\frac{s-1}{2} +\frac{1}{2\sigma}}]e^{\delta(1+|\xi-\eta|^2)^{\frac{1}{2\sigma}}}\Big).
		\end{align*}
	\end{lemm}
	\begin{lemm}\cite{hehuij}\label{trans}
		Let $\delta\geq 0,~\sigma\geq 1$ and $s>1+\frac d 2.$ Suppose that $u\in G^{\delta}_{\sigma,\frac{1}{2\sigma}}(\mathbb{R}),~v\in G^{\delta}_{\sigma,\frac{1}{2\sigma}}(\mathbb{R}).$ Then, we have
		\begin{align*}
			|\langle \Lambda^{s}e^{\delta{\Lambda}^{\frac{1}{\sigma}}}(u\cdot v),\Lambda^{s}e^{\delta{\Lambda}^{\frac{1}{\sigma}}}v\rangle|&\leq C\|\Lambda^{s}u\| \|\Lambda^{s}v\|^2+C\delta\Big(\|\Lambda^{s}e^{\delta{\Lambda}^{\frac{1}{\sigma}}}u\| \|\Lambda^{s+\frac{1}{\sigma}}e^{\delta{\Lambda}^{\frac{1}{\sigma}}}v\|^2\notag\\&~~~+\|\Lambda^{s+\frac{1}{\sigma}}e^{\delta{\Lambda}^{\frac{1}{\sigma}}}u\| \|\Lambda^{s+\frac{1}{\sigma}}e^{\delta{\Lambda}^{\frac{1}{\sigma}}}v\|\|\Lambda^{s}e^{\delta{\Lambda}^{\frac{1}{\sigma}}}v\|\Big),
		\end{align*}
	where $\langle\cdot,\cdot\rangle$ represents the scalar product of $L^2(\mathbb{R}^d)$, and $\|\cdot\|: =\|\cdot\|_{L^2}$, $\Lambda^s$ denotes for the inverse Fourier transform of $(1+|\xi|^2)^{\frac{s}{2}}\hat{f}.$
	\end{lemm}

\begin{lemm}\cite{hehuij}\label{inter}
For all $\delta\geq 0,~\sigma\geq1,~s\in\mathbb{R}$ and $l>0.$ Let $u\in G^{\delta}_{\sigma,\frac{1}{2\sigma}}(\mathbb{R}).$ Then we have the following estimate
\begin{align*}
	\|u\|_{G^{\delta}_{\sigma,s}}\leq\sqrt{e}\|u\|_{H^s}+(2\delta)^{\frac l 2}\|u\|_{G^{\delta}_{\sigma,s+\frac{l}{2\sigma}}}.
\end{align*}
\end{lemm}
We can now state the main result of this section.
\begin{theo}\label{global gevery}
Given $s>\frac 3 2,~\sigma\geq 1.$ Let $u_0\in G_{\sigma}(\mathbb{R}).$ Assume that $|\alpha|+\Gamma|+\|u_0\|_{H^s}+\frac{|\beta|}{3}\|u_0\|^2_{H^s}+\frac{|\gamma|}{4}\|u_0\|^3_{H^s}\leq \lambda \epsilon.$ Then there exists a unique global solution of system \eqref{u} in Gevery class $\sigma,$  namely, for any $t\geq 0,~u(t,\cdot)$ is of Gevery class $\sigma.$
\end{theo}
\begin{proof}
	We aim to get the global priori estimate of $u$  in the time-dependent space $G^{\delta(t)}_{\sigma,s}.$ One can apply the Fourier Galerkin approximating method to construct local solutions in $G^{\delta(t)}_{\sigma,s},$  globalize the
	result by the later estimate \eqref{g7}.
	
	According to the definition of the Gevery class, we see that $G^{\delta}_{\sigma,s}\hookrightarrow G^{\delta-\epsilon}_{\sigma,\infty}$ for any $\epsilon>0$. Notice that $u_0\in G_{\sigma}(\mathbb{R}).$ In general, we can suppose that $u_0\in G^1_{\sigma,s}.$ In the following, we will claim that there exists a $\delta(t)$ such that the solutions remain belong to  the Gevery class $G_{\sigma}.$ For any $t\in[0,T],$ it follows from the system \eqref{u} that
 \begin{align}\label{g}
 	\frac{d}{dt}\|u\|^2_{G^{\delta(t)}_{\sigma,s}}&=\frac{d}{dt}\int (1+|\xi|^2)^se^{2\delta(t)|\xi|^{\frac{1}{\sigma}}}\hat{u}(\xi)\bar{\hat{u}}(\xi)d\xi\notag\\
 	&=2\dot{\delta}(t)\int (1+|\xi|^2)^{s}|\xi|^{\frac{1}{\sigma}}e^{2\delta(t)|\xi|^{\frac{1}{\sigma}}}\hat{u}(\xi)\bar{\hat{u}}(\xi)d\xi\notag\\&~~~+\textbf{Re}\int (1+|\xi|^2)^se^{2\delta(t)|\xi|^{\frac{1}{\sigma}}}\Big(-\widehat{uu_x}-\Gamma\hat{u}-\lambda\hat{u}-\widehat{P_x}\Big)(\xi)\bar{\hat{u}}(\xi)d\xi,
 \end{align}
with \textbf{Re} stands the real part of a complex number.

Taking advantage of Lemma \ref{trans}, we get
\begin{align}\label{g1}
|\int (1+|\xi|^2)^se^{2\delta(t)|\xi|^{\frac{1}{\sigma}}}\widehat{uu_x}(\xi)\bar{\hat{u}}(\xi)d\xi|&=|\langle\Lambda ^se^{\delta(t)\Lambda^{\frac{1}{\sigma}}}(uu_x),\Lambda ^se^{\delta(t)\Lambda^{\frac{1}{\sigma}}}u\rangle|\notag\\
&\leq C\Big(\|u\|^3_{H^s}+\delta(t)\|u\|_{G^{\delta(t)}_{\sigma,s}}\|u\|^2_{G^{\delta(t)}_{\sigma,s+\frac{1}{2\sigma}}}\Big).
\end{align}
Likewise, one can get the following estimates
\begin{align}\label{g2}
	&|\int (1+|\xi|^2)^se^{2\delta(t)(1+|\xi|^2)^{\frac{1}{2\sigma}}}\widehat{-P_x}(\xi)\bar{\hat{u}}(\xi)d\xi|\notag\\&=|\langle\Lambda ^se^{\delta(t)\Lambda^{\frac{1}{\sigma}}}\Lambda^{-2}\Big(\partial_x(-h(u)-u^2-\frac 1 2 u_x^2)\Big),\Lambda ^se^{\delta(t)\Lambda^{\frac{1}{\sigma}}}u\rangle|\notag\\&\leq\|\Lambda ^{s-1}e^{\delta(t)\Lambda^{\frac{1}{\sigma}}}\Big(\partial_x(-h(u)-u^2-\frac 1 2 u_x^2)\Big)\|\|\Lambda ^se^{\delta(t)\Lambda^{\frac{1}{\sigma}}}u\|\notag\\
&\leq C\Big(\|u\|^2_{G^{\delta(t)}_{\sigma,s}}+\|u\|^3_{G^{\delta(t)}_{\sigma,s}}+\|u\|^4_{G^{\delta(t)}_{\sigma,s}}+\|u\|^5_{G^{\delta(t)}_{\sigma,s}}\Big).
\end{align}
Combining Lemma \ref{inter} with $l=1,~l=\frac 2 3,~l=\frac 1 2$ and $l=\frac 2 5,$ we obtain
\begin{align}\label{g4}
&\|u\|^2_{G^{\delta(t)}_{\sigma,s}}\leq C\Big(\|u\|^2_{H^s}+\delta(t)\|u\|^2_{G^{\delta(t)}_{\sigma,s+\frac{1}{2\sigma}}}\Big),\notag\\
&\|u\|^3_{G^{\delta(t)}_{\sigma,s}}\leq C\Big(\|u\|^3_{H^s}+\delta(t)\|u\|^3_{G^{\delta(t)}_{\sigma,s+\frac{1}{3\sigma}}}\Big)\leq C\Big(\|u\|^3_{H^s}+\delta(t)\|u\|_{G^{\delta(t)}_{\sigma,s}}\|u\|^2_{G^{\delta(t)}_{\sigma,s+\frac{1}{2\sigma}}}\Big),\notag\\
&\|u\|^4_{G^{\delta(t)}_{\sigma,s}}\leq C\Big(\|u\|^4_{H^s}+\delta(t)\|u\|^4_{G^{\delta(t)}_{\sigma,s+\frac{1}{4\sigma}}}\Big)\leq C\Big(\|u\|^4_{H^s}+\delta(t)\|u\|^2_{G^{\delta(t)}_{\sigma,s}}\|u\|^2_{G^{\delta(t)}_{\sigma,s+\frac{1}{2\sigma}}}\Big),\notag\\
&\|u\|^5_{G^{\delta(t)}_{\sigma,s}}\leq C\Big(\|u\|^5_{H^s}+\delta(t)\|u\|^5_{G^{\delta(t)}_{\sigma,s+\frac{1}{5\sigma}}}\Big)\leq C\Big(\|u\|^5_{H^s}+\delta(t)\|u\|^3_{G^{\delta(t)}_{\sigma,s}}\|u\|^2_{G^{\delta(t)}_{\sigma,s+\frac{1}{2\sigma}}}\Big).
\end{align}
By the same token, it yields
\begin{align}\label{g5}
&|\int (1+|\xi|^2)^se^{2\delta(t)(1+|\xi|^2)^{\frac{1}{2\sigma}}}(-\Gamma\hat{u}-\lambda\hat{u})(\xi)\bar{\hat{u}}(\xi)d\xi| \leq \Big(\|u\|^2_{H^s}+C\delta(t)\|u\|^2_{G^{\delta(t)}_{\sigma,s+\frac{1}{2\sigma}}}\Big).
\end{align}
Plugging \eqref{g1}-\eqref{g5} into \eqref{g}, we obtain
\begin{align*}
	\frac 1 2 \frac{d}{dt}\|u\|^2_{G^{\delta(t)}_{\sigma,s}}&\leq \Big(\dot{\delta}(t)+C\delta(t)(1+\|u\|_{G^{\delta(t)}_{\sigma,s}})^3\Big)\|u\|^2_{G^{\delta(t)}_{\sigma,s+\frac{1}{2\sigma}}}+C(\underbrace{1+\|u\|_{H^s}}_{b(t)})^5,
\end{align*}
which implies
\begin{align*}
	\frac 1 2 \frac{d}{dt}(1+\|u\|_{G^{\delta(t)}_{\sigma,s}})^2&\leq \Big(\dot{\delta}(t)+C\delta(t)(1+\|u\|_{G^{\delta(t)}_{\sigma,s}})^3\Big)(1+\|u\|_{G^{\delta(t)}_{\sigma,s+\frac{1}{2\sigma}}})^2+Cb^5(t).
\end{align*}
For any $t\in [0,T_0],$ we assume that
\begin{align}\label{f1}
	(1+\|u\|_{G^{\delta(t)}_{\sigma,s}})^2\leq 4f^2(t),
\end{align}
where $f^2(t):=2(1+\|u_0\|_{G^{\delta_0}_{\sigma,s}})^2+2C\int_0^tb^5(t')dt'$
and take
\begin{align}\label{G2}
	\dot{\delta}(t)=-8C\delta(t)f^3(t).
\end{align}
Noting that $B^s_{2,2}=H^s,$
 Lemma \ref{hsglobal} guarantees that the existence of global classical solution $u\in \mathcal{C}(\mathbb{R}^+;H^s).$ Hence, it follows  from  \eqref{G2} that
\begin{align}\label{g9}
	\delta(t)=\delta_0{\rm \exp}\Big(-8C\int_0^t f^3(t')dt'\Big)>0,~\forall~t\in [0,\infty),
\end{align}
where $0<\delta_0<1$.\\
 Combining \eqref{f1} and \eqref{G2}, we get for any $[0,T_0]$
\begin{align}\label{g7}
	(1+\|u(t)\|_{G^{\delta(t)}_{\sigma,s}})^2\leq f^2(t).
\end{align}
Applying Theorem \ref{local}, we have the unique  global solution on time interval $[0,T_0].$ Moreover, we get
\begin{align}\label{G3}
	(1+\|u(T_0)\|_{G^{\delta(T_0)}_{\sigma,s}})^2\leq f^2(T_0).
\end{align}
Using \eqref{f1}-\eqref{G2} and \eqref{G3}  again, then there exists a $T_1>T_0$ such that $(1+\|u\|_{G^{\delta(t)}_{\sigma,s}})^2\leq 2f^2(t)$ on $[T_0,T_1].$ Therefore, we have $\sup_{t\in [T_0,T_1]} (1+\|u\|_{G^{\delta(t)}_{\sigma,s}})^2\leq f^2(t).$ Moreover, we conclude that
$$\sup_{t\in [0,T_1]}	(1+\|u(t)\|_{G^{\delta(t)}_{\sigma,s}})^2\leq f^2(t).$$
  Repeating the bootstrap argument, we see that the global existence of strong solution of  the system \eqref{u}.
\end{proof}
		\noindent\textbf{Acknowledgements.}
	This work was partially supported by NNSFC (Grant No. 12171493), FDCT (Grant No. 0091/2018/A3), the Guangdong Special Support Program (Grant No.8-2015).

\end{document}